\documentclass[final,1p,times]{my_elsarticle}

\usepackage{amssymb}
\usepackage{amsmath}
\usepackage{amsthm}

\newcommand{\R}{\mathbb{R}}
\newcommand{\Z}{\mathbb{Z}}
\newcommand{\T}{\mathbb{T}}

\newtheorem{theorem}{Theorem}

\newtheorem{lemma}[theorem]{Lemma}

\newtheorem{example}{Example}
\newtheorem{definition}[theorem]{Definition}

\def\di{\displaystyle}

\begin{document}

\begin{frontmatter}

\title{Differential, integral, and variational delta-embeddings\\
of Lagrangian systems}

\author[jc,jc2]{Jacky Cresson}
\ead{jacky.cresson@univ-pau.fr}

\author[abm,cidma]{Agnieszka B. Malinowska}
\ead{abmalinowska@ua.pt}

\author[cidma,dfmt]{Delfim F. M. Torres}
\ead{delfim@ua.pt}


\address[jc]{Laboratoire de Math\'{e}matiques Appliqu\'{e}es de Pau,
Universit\'{e} de Pau, Pau, France}

\address[jc2]{Institut de M\'ecanique C\'eleste et de Calcul
des \'Eph\'em\'erides, Observatoire de Paris, Paris, France}

\address[abm]{Faculty of Computer Science,
Bia{\l}ystok University of Technology,
15-351 Bia\l ystok, Poland}

\address[cidma]{R\&D Unit CIDMA, University of Aveiro, 3810-193 Aveiro, Portugal}

\address[dfmt]{Department of Mathematics,
University of Aveiro,
3810-193 Aveiro, Portugal}


\begin{abstract}
We introduce the differential, integral, and variational delta-embeddings.
We prove that the integral delta-embedding of the Euler--Lagrange equations
and the variational delta-embedding coincide on an arbitrary time scale.
In particular, a new coherent embedding for the discrete calculus of variations
that is compatible with the least action principle is obtained.
\end{abstract}


\begin{keyword}
coherence \sep  embedding \sep least action principle
\sep discrete calculus of variations
\sep difference Euler--Lagrange equations.

\medskip

\MSC[2010] 34N05 \sep 49K05 \sep 49S05.

\end{keyword}

\end{frontmatter}


\section{Introduction}

An ordinary differential equation is usually
given in differential form, \textrm{i.e.},
\begin{equation*}
\frac{dx(t)}{dt} =f\left(t,x(t)\right),
\quad  t\in [a,b] , \quad x(t) \in \R^n .
\end{equation*}
However, one can also consider the integral form of the equation:
\begin{equation*}
x(t) =x(a)+\di\int_a^t f(s,x(s)) ds , \quad t \in [a,b].
\end{equation*}
The differential form is related to dynamics via the time derivative.
The integral form is useful for proving existence and unicity
of solutions or to study analytical properties of solutions.

In order to give a meaning to a differential equation over a new set
(\textrm{e.g.}, stochastic processes, non-differentiable functions
or discrete sets) one can use the differential or the integral form.
In general, these two generalizations do not give the same object. In the differential case,
we need to extend first the time derivative. As an example, we can look to Schwartz's distributions
\cite{sc} or backward/forward finite differences in the discrete case. Using the new derivative
one can then generalize differential operators and then differential equations of arbitrary order.
In the integral case, one need to give a meaning to the integral over the new set.
This strategy is for example used by It\^o \cite{ito} in order to define stochastic
differential equations, defining first stochastic integrals.
In general, the integral form imposes less constraints on the underlying objects.
This is already in the classical case, where we need a differentiable function
to write the differential form but only continuity or weaker regularity
to give a meaning to the integral form.

The notion of embedding introduced in \cite{cd} is an algebraic procedure providing
an extension of classical differential equations over an arbitrary vector space.
Embedding is based on the differential formulation of the equation.
This formalism was developed in the framework of
fractional equations \cite{cr1}, stochastic processes \cite{cd},
and non-differentiable functions \cite{cg}. Recently, it has been extended to discrete sets in order
to discuss discretization of differential equations and numerical schemes \cite{BCGI}.

In this paper, we define an embedding containing the discrete as well as the continuous
case using the time scale calculus. We use the proposed embedding
in order to define time scale extensions of ordinary differential equations
both in differential and integral forms.

Of particular importance for many applications in physics and mathematics
is the case of Lagrangian systems governed by an Euler--Lagrange equation.
Lagrangian systems possess a variational structure,
\textrm{i.e.}, their solutions correspond to critical points of
a functional and this characterization does not depend on the coordinates system.
This induces strong constraints on solutions, for example the conservation
of energy for autonomous classical Lagrangian systems.
That is, if the Lagrangian does not depend explicitly on $t$,
then the energy is constant along physical trajectories.
We use the time scale embedding in order to provide an analogous
of the classical Lagrangian functional on an arbitrary time scale.
By developing the corresponding calculus of variations,
one then obtain the Euler--Lagrange equation.
This extension of the original Euler--Lagrange equation passing
by the time scales embedding of the functional is called a time scale
{\it variational embedding}. These extensions are known under the terminology
of {\it variational integrators} in the discrete setting.

We then have three ways to extend a given ordinary differential equation:
differential, integral or variational embedding. All these extensions
are a priori different. The coherence problem introduced in \cite{cd},
in the context of the stochastic embedding, consider the problem of finding
conditions under which these extensions coincide.
Here we prove that the integral and variational embeddings are coherent
(see Theorem~\ref{thm:mr}). The result is new and interesting even
in the discrete setting, providing a new form of the Euler--Lagrange
difference equation (see \eqref{neleq}) that is compatible with the least action principle.


\section{Note on the notation used}

We denote by $f$ or $t \rightarrow f(t)$ a function,
and by $f(t)$ the value of the function at point $t$.
Along all the text we consistently use square brackets for the arguments of
operators and round brackets for the arguments of all the other type of functions.
Functionals are denoted by uppercase letters in calligraphic mode.
We denote by $D$ the usual differential operator
and by $\partial_i$ the operator of partial differentiation
with respect to the $i$th variable.


\section{Reminder about the time scale calculus}
\label{sec:rtsc}

The reader interested on the calculus on time scales
is refereed to the book \cite{book:ts:2001}. Here
we just recall the necessary concepts and fix some notations.

A nonempty closed subset of $\R$ is called a time scale
and is denoted by $\T$. Thus, $\mathbb{R}$,
$\mathbb{Z}$, and $\mathbb{N}$, are trivial examples of time scales.
Other examples of time scales are: $[-2,4] \bigcup \mathbb{N}$,
$h\mathbb{Z}:=\{h z \, | \, z \in \mathbb{Z}\}$ for some
$h>0$, $q^{\mathbb{N}_0}:=\{q^k \, | \, k \in \mathbb{N}_0\}$ for some
$q>1$, and the Cantor set. We assume that a time scale $\mathbb{T}$
has the topology that it inherits from the real numbers with the
standard topology.

The forward jump $\sigma :\T \rightarrow \T$ is defined by
$\sigma(t)=\inf{\{s\in\mathbb{T}:s>t}\}$
for all $t\in\mathbb{T}$, while the \emph{backward jump}
$\rho:\mathbb{T}\rightarrow\mathbb{T}$ is defined by
$\rho(t)=\sup{\{s\in\mathbb{T}:s<t}\}$ for all
$t\in\mathbb{T}$, where $\inf\emptyset=\sup\mathbb{T}$
(\textrm{i.e.}, $\sigma(M)=M$ if $\mathbb{T}$ has a maximum $M$)
and $\sup\emptyset=\inf\mathbb{T}$ (\textrm{i.e.}, $\rho(m)=m$ if
$\mathbb{T}$ has a minimum $m$). The \emph{graininess function}
$\mu:\mathbb{T}\rightarrow[0,\infty)$ is defined by
$\mu(t)=\sigma(t)-t$ for all $t\in\mathbb{T}$.

\begin{example}
If $\mathbb{T} = \mathbb{R}$, then $\sigma(t) = \rho(t) = t$
and $\mu(t)=0$. If $\mathbb{T} = h\mathbb{Z}$, then
$\sigma(t) = t + h$, $\rho(t) = t - h$, and
$\mu(t)=h$. On the other hand, if
$\mathbb{T} = q^{\mathbb{N}_0}$, where $q>1$ is a fixed real number,
then we have $\sigma(t) = q t$, $\rho(t) = q^{-1} t$,
and $\mu(t)= (q-1) t$.
\end{example}

In order to introduce the definition of delta derivative, we
define a new set $\mathbb{T}^\kappa$ which is derived from $\mathbb{T}$
as follows: if  $\mathbb{T}$ has a left-scattered maximal point $M$,
then $\mathbb{T}^\kappa := \mathbb{T}\setminus\{M\}$; otherwise,
$\mathbb{T}^\kappa := \mathbb{T}$. In general, for $r \ge 2$,
$\mathbb{T}^{\kappa^r}:=\left(\mathbb{T}^{\kappa^{r-1}}\right)^\kappa$.
Similarly, if  $\mathbb{T}$ has a right-scattered minimum $m$,
then we define $\mathbb{T}_\kappa := \mathbb{T}\setminus\{m\}$; otherwise,
$\mathbb{T}_\kappa := \mathbb{T}$. Moreover, we define
$\mathbb{T}^{\kappa}_{\kappa}:=\mathbb{T}^{\kappa}\cap \mathbb{T}_{\kappa}$.

\begin{definition}
\label{def:de:dif}
We say that a function $f:\mathbb{T}\rightarrow\mathbb{R}$
is \emph{delta differentiable} at $t\in\mathbb{T}^{\kappa}$ if there
exists a number $\Delta[f](t)$ such that for all $\varepsilon>0$
there is a neighborhood $U$ of $t$ such that
$$
\left|f(\sigma(t))-f(s)-\Delta[f](t)(\sigma(t)-s)\right|
\leq\varepsilon|\sigma(t)-s|,\mbox{ for all $s\in U$}.
$$
We call $\Delta[f](t)$ the \emph{delta derivative} of $f$ at $t$ and we say
that $f$ is \emph{delta differentiable} on $\mathbb{T}^{\kappa}$
provided $\Delta[f](t)$ exists for all $t\in\mathbb{T}^{\kappa}$.
\end{definition}

\begin{example}
\label{ex:der:QC}
If $\mathbb{T} = \mathbb{R}$, then $\Delta[f](t) = D[f](t)$, \textrm{i.e.},
the delta derivative coincides with the usual one.
If $\mathbb{T} = h\mathbb{Z}$, then $\Delta[f](t)
= \frac{1}{h}\left(f(t+h) - f(t)\right)
=: \Delta_+[f](t)$, where $\Delta_+$ is the usual forward difference
operator defined by the last equation. If $\mathbb{T} = q^{\mathbb{N}_0}$, $q>1$,
then $\Delta[f](t)=\frac{f(q t)-f(t)}{(q-1) t}$,
\textrm{i.e.}, we get the usual derivative of quantum calculus.
\end{example}

A function $f:\mathbb{T}\rightarrow\mathbb{R}$ is called \emph{rd-continuous}
if it is continuous at right-dense points and if its left-sided limit exists
at left-dense points. We denote the set of all rd-continuous functions by
$C_{\textrm{rd}}(\mathbb{T}, \mathbb{R})$ and the set of all delta differentiable
functions with rd-continuous derivative by $C_{\textrm{rd}}^1(\mathbb{T}, \mathbb{R})$.
It is known (see \cite[Theorem~1.74]{book:ts:2001}) that  rd-continuous functions possess
a \emph{delta antiderivative}, \textrm{i.e.}, there exists a function $\xi$
with $\Delta[\xi]=f$, and in this case the \emph{delta integral}
is defined by $\int_{c}^{d}f(t)\Delta t=\xi(d)-\xi(c)$ for all $c,d\in\mathbb{T}$.

\begin{example}
Let $a, b \in \mathbb{T}$ with $a < b$.
If $\mathbb{T} = \mathbb{R}$, then
$\int_{a}^{b}f(t)\Delta t = \int_{a}^{b} f(t) dt$,
where the integral on the right-hand side is the
classical Riemann integral. If $\mathbb{T} = h\mathbb{Z}$,
then $\int_{a}^{b} f(t)\Delta t = \sum_{k=\frac{a}{h}}^{\frac{b}{h}-1}hf(kh)$.
If $\mathbb{T} = q^{\mathbb{N}_0}$, $q>1$, then
$\int_{a}^{b} f(t)\Delta t = (1 - q) \sum_{t \in [a,b)} t f(t)$.
\end{example}

The delta integral has the following properties:
\begin{itemize}

\item[(i)] if $f\in C_{rd}$ and $t \in \mathbb{T}$, then
\begin{equation*}
\int_t^{\sigma(t)} f(\tau)\Delta\tau=\mu(t)f(t)\, ;
\end{equation*}

\item[(ii)]if $c,d\in\mathbb{T}$ and $f$ and $g$
are delta-differentiable, then the following formulas
of integration by parts hold:
\begin{equation*}
\int_{c}^{d} f(\sigma(t))\Delta[g](t)\Delta t
=\left.(fg)(t)\right|_{t=c}^{t=d}
-\int_{c}^{d} \Delta[f](t) g(t)\Delta t,
\end{equation*}
\begin{equation}
\label{int:par:delta}
\int_{c}^{d} f(t) \Delta[g](t)\Delta t
=\left.(fg)(t)\right|_{t=c}^{t=d}
-\int_{c}^{d} \Delta[f](t) g(\sigma(t)) \Delta t.
\end{equation}

\end{itemize}


\section{Time scale embeddings and evaluation operators}
\label{sec:tse}

Let $\T$ be a bounded time scale with $a:= \min \T$ and $b:= \max \T$.
We denote by $C([a,b];\R)$ the set of continuous
functions $x : [a,b] \rightarrow \R$. As introduced
in Section~\ref{sec:rtsc}, by $C_{\textrm{rd}}(\T,\R)$
we denote the set of all real valued rd-continuous functions
defined on $\T$, and by $C_{\textrm{rd}}^1(\mathbb{T}, \mathbb{R})$
the set of all delta differentiable functions with rd-continuous derivative.

A time scale embedding is given by specifying:
\begin{itemize}
\item A mapping $\iota : C([a,b],\R) \rightarrow C_{\textrm{rd}}(\T,\R )$;
\item An operator $\delta : C^1([a,b],\R) \rightarrow C_{\textrm{rd}}^1(\T^\kappa,\R)$,
called a generalized derivative;
\item An operator $J: C([a,b],\R)\rightarrow C_{\textrm{rd}}(\T,\R)$,
called a generalized integral operator.
\end{itemize}

We fix the following embedding:

\begin{definition}[Time scale embedding]
The mapping $\iota$ is obtained by restriction of functions to $\T$.
The operator $\delta$ is chosen to be the $\Delta$ derivative, and the operator
$J$ is given by the $\Delta$ integral as follows:
$$
\delta[x](t) := \Delta[x](t) \, ,  \quad
J[x](t) := \di\int_a^{\sigma (t)} x(s) \Delta s\, .
$$
\end{definition}

\begin{definition}[Evaluation operator]
Let $f :\R \rightarrow \R$ be a continuous function.
We denote by $\widetilde{f}$ the operator associated
to $f$ and defined by
\begin{equation}
\label{evaluation}
\widetilde{f} : \left .
\begin{array}{ccc}
C (\R ,\R ) & \longrightarrow & C (\R ,\R )\\
x & \mapsto & \widetilde{f}[x] := t \rightarrow f(x(t)) .
\end{array}
\right .
\end{equation}
The operator $\widetilde{f}$ given by \eqref{evaluation}
is called the \emph{evaluation operator} associated with $f$.
\end{definition}

The definition of evaluation operator is easily extended in various ways.
We give in Definition~\ref{def:L:eval:oper} a special evaluation operator
that naturally arises in the study of problems of the calculus of variations
and respective Euler--Lagrange equations (\textrm{cf.} Section~\ref{sec:cv}).

\begin{definition}[Lagrangian operator]
\label{def:L:eval:oper}
Let $L : [a,b] \times \R \times \R \rightarrow \R$ be a $C^1$ function defined for all
$(t,x,v) \in [a,b] \times \R^2$ by $L(t,x,v)\in \R$. The \emph{Lagrangian operator}
$\widetilde{L} : C^1([a,b] ,\R) \rightarrow C^1([a,b] ,\R)$
associated with $L$ is the evaluation operator defined by
$\widetilde{L}[x] := t \rightarrow L\left(t,x(t),D[x](t)\right)$.
\end{definition}

We consider ordinary differential equations of the form
\begin{equation*}
O[x](t) =0, \quad t \in [a,b],
\end{equation*}
where $x\in C^n(\R,\R)$ and $O$ is a differential operator of order $n$,
$n \ge 1$, given by
\begin{equation}
\label{operator}
O=\di\sum_{i=0}^n \widetilde{a_i} \cdot \left( D^i \circ \widetilde{b_i} \right),
\end{equation}
where $(\widetilde{a_i})$ (resp. $(\widetilde{b_i})$)
is the family of evaluation operators associated
to a family of functions $(a_i)$ (resp. $(b_i)$),
and $D^i$ is the derivative of order $i$, \textrm{i.e.},
$D^i = \frac{d^i}{dt^i}$.
Differential operators of form \eqref{operator} play
a crucial role when dealing with Euler--Lagrange equations.

We are now ready to define the time scale embedding
of evaluation and differential operators.

\begin{definition}[Time scale embedding of evaluation operators]
Let $f : \R \rightarrow \R$ be a continuous function and $\widetilde{f}$
the associated evaluation operator. The time scale embedding $\widetilde{f}_\T$
of $\widetilde{f}$ is the extension of $\widetilde{f}$ to $C_{rd}(\T ,\R )$:
\begin{equation*}
\widetilde{f}_{\T} : \left .
\begin{array}{ccc}
C_{rd}(\T ,\R ) & \longrightarrow & C_{rd}(\T ,\R )\\
x & \mapsto & \widetilde{f}_\T[x] :=  t \rightarrow f(x(t)) .
\end{array}
\right .
\end{equation*}
\end{definition}

Next definition gives the time scale embedding
of the differential operator \eqref{operator}.

\begin{definition}[Time scale embedding of differential operators]
\label{eq:diff:emb}
The time scale embedding of the differential operator \eqref{operator} is defined by
\begin{equation*}
O_{\Delta} =\di\sum_{i=0}^n \widetilde{a_i}_\T \cdot \left( {\Delta^i} \circ \widetilde{b_i}_\T \right).
\end{equation*}
\end{definition}

The two previous definitions are sufficient to define the time scale embedding
of a given ordinary differential equation.

\begin{definition}[Time scale embedding of differential equations]
\label{embeddingdofferential}
The delta-differential embedding of an ordinary differential equation
$O[x]=0$, $x \in C^n([a,b],\R)$, is given by
$O_{\Delta}[x] =0$, $x\in C^n_{rd}(\T^{\kappa^n} ,\R)$.
\end{definition}

In order to define the delta-integral and variational embeddings
(see Sections~\ref{sec:dealp}, \ref{sec:intgral} and \ref{sec:mr})
we need to know how to embed an integral functional.

\begin{definition}[Time scale embedding of integral functionals]
\label{embeddingfunctional}
Let $L : [a,b] \times \R^2 \rightarrow \R$ be a continuous function and $\mathcal{L}$
the functional defined by
$$
\mathcal{L}(x) = \int_a^t L\left(s,x(s),D[x](s)\right) ds
= \int_a^t \widetilde{L}[x](s) ds.
$$
The time scale embedding $\mathcal{L}_{\Delta}$ of $\mathcal{L}$ is given by
\begin{equation*}
\mathcal{L}_{\Delta}(x) =\int_a^{\sigma(t)} L\left(s,x(s),\Delta[x](s)\right) \Delta s
= \int_a^{\sigma(t)} \widetilde{L}_\T[x](s) \Delta s \, .
\end{equation*}
\end{definition}


\section{Calculus of variations}
\label{sec:cv}

The classical variational functional $\mathcal{L}$ is defined by
\begin{equation}
\label{lagrangianfunctional}
\mathcal{L}(x) =\int_a^b L\left(t,x(t),D[x](t)\right) dt ,
\end{equation}
where $L :[a,b] \times \R \times \R \rightarrow \R$
is a smooth real valued function called the Lagrangian
(see, \textrm{e.g.}, \cite{vanBrunt}).
Functional \eqref{lagrangianfunctional} can be written,
using the Lagrangian operator $\widetilde{L}$ (Definition~\ref{def:L:eval:oper}),
in the following equivalent form:
\begin{equation*}
\mathcal{L}(x) = \int_a^{b} \widetilde{L}[x](t) dt .
\end{equation*}
The Euler--Lagrange equation associated
to \eqref{lagrangianfunctional} is given
(see, \textrm{e.g.}, \cite{vanBrunt}) by
\begin{equation}
\label{eq:class:EL:pre}
D\left[ \tau \rightarrow
\partial_3[L]\left(\tau,x(\tau), D[x](\tau)\right)
\right](t) -\partial_2[L]\left(t,x(t),D[x](t)\right) =0 ,
\end{equation}
$t \in [a,b]$, which we can write, equivalently, as
\begin{equation*}
\left(D \circ \widetilde{\partial_3[L]}\right)[x](t)
- \widetilde{\partial_2[L]}[x](t) =0 .
\end{equation*}
Still another way to write the Euler--Lagrange equation
consists in introducing the differential operator $\mbox{\rm EL}_L$,
called the {\it Euler--Lagrange operator}, given by
\begin{equation*}
\mbox{\rm EL}_L := D \circ \widetilde{\partial_3[L]} - \widetilde{\partial_2[L]} \, .
\end{equation*}
We can then write the Euler--Lagrange equation
simply as $\mbox{\rm EL}_L[x](t)=0$, $t \in [a,b]$.


\section{Delta-differential embedding of the Euler--Lagrange equation}
\label{sec:deEL}

By Definition~\ref{eq:diff:emb}, the time scale delta
embedding of the Euler--Lagrange operator $\mbox{\rm EL}_L$
gives the new operator
\begin{equation*}
\left(\mbox{\rm EL}_L\right)_{\Delta}
:= \Delta \circ  \left(\widetilde{\partial_3[L]}\right)_\T
- \left(\widetilde{\partial_2[L]}\right)_\T.
\end{equation*}
As a consequence, we have the following lemma.

\begin{lemma}[Delta-differential embedding of the Euler--Lagrange equation]
The delta-differential embedding of the Euler--Lagrange equation
is given by $(\mbox{\rm EL}_L)_{\Delta}[x](t) = 0$,
$t \in \T^{\kappa^2}$, \textrm{i.e.},
\begin{equation}
\label{eq:d:EL}
\left(\Delta \circ  \widetilde{\partial_3[L]}_\T\right)[x](t)
- \widetilde{\partial_2[L]}_\T[x](t) = 0
\end{equation}
for any $t \in \T^{\kappa^2}$.
\end{lemma}

In the discrete case $\T =[a,b]\cap h\Z$, we obtain from \eqref{eq:d:EL}
the well-known discrete version of the Euler--Lagrange equation,
often written as
\begin{equation}
\label{eq:d:d:EL}
\Delta_+ \circ \frac{\partial L}{\partial v}\left(t,x(t),\Delta_+ x(t)\right)
-\frac{\partial L}{\partial x}\left(t,x(t),\Delta_+ x(t)\right)= 0 ,
\end{equation}
$t \in \T^{\kappa^2}$, where $\Delta_+ f(t) = \frac{f(t+h)-f(t)}{h}$.
The important point to note here is that from the numerical point of view,
equation \eqref{eq:d:d:EL} does not provide a good scheme. Let us see
a simple example.

\begin{example}
\label{ex1}
Consider the Lagrangian $L(t,x,v) = \frac{1}{2} v^2 - U(x)$, where $U$
is the potential energy and $(t,x,v) \in [a,b] \times \R\times \R$.
Then the Euler--Lagrange equation \eqref{eq:d:d:EL} gives
\begin{equation}
\label{alg1}
\frac{x_{k+2}-2 x_{k+1}+x_k}{h^2} + \frac{\partial U}{\partial x}(x_k) = 0,
\quad k = 0, \ldots, N-2,
\end{equation}
where $N = \frac{b-a}{h}$ and $x_k = x(a+hk)$.
This numerical scheme is of order one, meaning that we make
an error of order $h$ at each step, which is of course not good.
\end{example}

In the next section we show an alternative
Euler--Lagrange to \eqref{eq:d:d:EL} that leads
to more suitable numerical schemes.
As we shall see in Section~\ref{sec:mr}, this comes
from the fact that the embedded Euler--Lagrange equation \eqref{eq:d:EL}
is not coherent, meaning that it does not preserve the variational structure.
As a consequence, the numerical scheme \eqref{alg1} is not {\it symplectic}
in contrast to the flow of the Lagrangian system (see \cite{Arnold1979}).
In particular, the numerical scheme \eqref{alg1} dissipates artificially energy
(see \cite[Fig.~1, p.~364]{Marsden-West2001}).


\section{Discrete variational embedding}
\label{sec:dealp}

The time scale embedding can be also used to define a delta analogue
of the variational functional \eqref{lagrangianfunctional}.
Using Definition~\ref{embeddingfunctional},
and remembering that $\sigma(b)=b$, the time scale embedding
of \eqref{lagrangianfunctional} is
\begin{equation}
\label{funct}
\mathcal{L}_{\Delta}(x)
=\int_a^b L\left(t,x(t) ,\Delta[x](t)\right)\, \Delta t
=\int_a^b \widetilde{L}_\T[x](t) \Delta t.
\end{equation}
A calculus of variations on time scales for functionals of type \eqref{funct}
is developed in Section~\ref{sec:mr}. Here we just emphasize
that in the discrete case $\T =[a,b]\cap h\Z$
functional \eqref{funct} reduces to the classical discrete Lagrangian functional
\begin{equation}
\label{dpcv}
\mathcal{L}_\Delta(x)= h \sum_{k=0}^{N-1} L\left(t_k,x_k,\Delta_+ x_k\right),
\end{equation}
where $N = \frac{b-a}{h}$, $x_k = x(a+hk)$ and
$\Delta_+ x_k = \frac{x_{k+1}-x_k}{h}$, and that
the Euler--Lagrange equation obtained by applying
the discrete variational principle to \eqref{dpcv}
takes the form
\begin{equation}
\label{eq:nabla:EL}
\Delta_- \circ \frac{\partial L}{\partial v}\left(t,x(t),\Delta_+ x(t)\right)
-\frac{\partial L}{\partial x}\left(t,x(t),\Delta_+ x(t)\right)= 0,
\end{equation}
$t \in \T^{\kappa}_{\kappa}$, where $\Delta_-$ is the backward finite difference operator
defined by $\Delta_- f(t) = \frac{f(t)-f(t-h)}{h}$ \cite{BCGI,lu}.

The numerical scheme corresponding to the
discrete variational embedding,
\textrm{i.e.}, to \eqref{eq:nabla:EL},
is called in the literature a {\it variational integrator} \cite{BCGI,lu}.
Next example shows that the variational integrator
associated with problem of Example~\ref{ex1}
is a better numerical scheme than \eqref{alg1}.

\begin{example}
\label{ex2}
Consider the same Lagrangian as in Example~\ref{ex1} of Section~\ref{sec:deEL}:
$L(t,x,v) = \frac{1}{2} v^2 - U(x)$, where $(t,x,v) \in [a,b] \times \R \times \R$.
The Euler--Lagrange equation \eqref{eq:nabla:EL} can be written as
$$
\frac{x_{k+1}-2 x_{k}+x_{k-1}}{h^2} + \frac{\partial U}{\partial x}(x_k) = 0,
\quad k = 1, \ldots, N-1,
$$
where $N = \frac{b-a}{h}$ and $x_k = x(a+hk)$.
This numerical scheme possess very good properties. In particular, it is easily seen
that the order of approximation is now two and not of order one as in Example~\ref{ex1}.
\end{example}

We remark that the form of $\mathcal{L}_{\Delta}$ given by \eqref{funct}
is not the usual one in the literature of time scales (see \cite{bartos:torres,malina:T,naty}
and references therein). Indeed, in the literature of the calculus of variations
on time scales, the following version of the Lagrangian functional is studied:
\begin{equation}
\label{eq:std:cv:t}
\mathcal{L}_{\T}^{\rm usual}(x)
= \int_a^b L\left(t,x(\sigma(t)),\Delta[x](t)\right) \Delta t.
\end{equation}
However, the composition of $x$ with the forward jump $\sigma$
found in \eqref{eq:std:cv:t} seems not natural from the point of view of embedding.


\section{Delta-integral embedding of the Euler--Lagrange equation in integral form}
\label{sec:intgral}

We begin by rewriting the classical Euler--Lagrange equation
into integral form. Integrating \eqref{eq:class:EL:pre} we obtain that
\begin{equation}
\label{eq:class:EL:if}
\partial_3[L]\left(t,x(t),D[x](t)\right)
= \int_a^{t} \partial_2[L]\left(\tau,x(\tau),D[x](\tau)\right) d\tau + c,
\end{equation}
for some constant $c$ and all $t \in [a,b]$ or, using
the evaluation operator,
\begin{equation*}
\widetilde{\partial_3[L]}[x](t)
= \int_a^{t} \widetilde{\partial_2[L]}[x](\tau) d\tau + c.
\end{equation*}
Using Definition~\ref{embeddingfunctional}, we obtain the delta-integral
embedding of the classical Euler--Lagrange equation \eqref{eq:class:EL:if}.

\begin{lemma}[Delta-integral embedding of the Euler--Lagrange equation in integral form]
The delta-integral embedding of the Euler--Lagrange equation \eqref{eq:class:EL:if} is given by
\begin{equation}
\label{eq:class:EL:if:ts}
\partial_3[L]\left(t,x(t),\Delta[x](t)\right)
= \int_a^{\sigma(t)} \partial_2[L]\left(\tau,x(\tau), \Delta[x](\tau)\right) \Delta \tau + c
\end{equation}
or, equivalently, as
\begin{equation*}
\widetilde{\partial_3[L]}_\T[x](t)
= \int_a^{\sigma(t)} \widetilde{\partial_2[L]}_\T[x](\tau) \Delta \tau + c ,
\end{equation*}
where $c$ is a constant and $t \in \T^\kappa$.
\end{lemma}

Note that in the particular case $\T=[a,b]\cap h\Z$
equation \eqref{eq:class:EL:if:ts} gives the
discrete Euler--Lagrange equation
\begin{equation}
\label{neleq}
\frac{\partial L}{\partial v}\left(t_k,x_k,\frac{x_{k+1}-x_k}{h}\right)
= h \sum_{i=0}^{k} \frac{\partial L}{\partial x}\left(t_i,x_i, \frac{x_{i+1}-x_i}{h}\right)  + c ,
\end{equation}
where $t_i = a + h i$, $i = 0, \ldots, k$, $x_i = x(t_i)$, and $k = 0, \ldots, N-1$.
This numerical scheme is different from \eqref{eq:d:d:EL} and \eqref{eq:nabla:EL},
and has not been discussed before in the literature
with respect to embedding and coherence. This is done
in Section~\ref{sec:mr}.


\section{The delta-variational embedding and coherence}
\label{sec:mr}

Our next theorem shows that equation \eqref{eq:class:EL:if:ts}
can be also obtained from the least action principle.
In other words, Theorem~\ref{thm:mr} asserts that the delta-integral
embedding of the classical Euler--Lagrange equation in integral form
\eqref{eq:class:EL:if} and the delta-variational embedding are coherent.

\begin{theorem}
\label{thm:mr}
If $\hat{x}$ is a local minimizer or maximizer to \eqref{funct}
subject to the boundary conditions $x(a)=x_a$ and $x(b)=x_b$,
then $\hat{x}$ satisfies the Euler--Lagrange equation
\eqref{eq:class:EL:if:ts} for some constant $c$ and all $t \in {\T}^\kappa$.
\end{theorem}

\begin{proof}
Suppose that $\mathcal{L}_{\Delta}$ has a weak local extremum at
$\hat{x}$. Let $x = \hat{x} + \varepsilon h$,
where $\varepsilon\in \mathbb{R}$ is a small parameter,
$h\in C^1_{rd}$ such that $h(a) =h(b)=0$. We consider
$$
\phi(\varepsilon) := \mathcal{L}_{\Delta}(\hat{x} + \varepsilon h)
= \int_a^b L\left(t,\hat{x}(t)+\varepsilon h(t),
\Delta[\hat{x}](t)+ \varepsilon \Delta[h](t)\right) \Delta t.
$$
A necessary condition for $\hat{x}$ to be an extremizer is
given by
\begin{equation}
\label{eq:FT}
\left.\phi'(\varepsilon)\right|_{\varepsilon=0} = 0
\Leftrightarrow \int_a^b \Bigl( \partial_2[L]\left(t,\hat{x}(t),\Delta[\hat{x}](t)\right) h(t)
+ \partial_3[L]\left(t,\hat{x}(t),\Delta[\hat{x}](t)\right) \Delta[h](t)\Bigr) \Delta t = 0.
\end{equation}
The integration by parts formula \eqref{int:par:delta} gives
\begin{multline*}
\int_a^b \partial_2[L]\left(t,\hat{x}(t),\Delta[\hat{x}](t)\right) h(t)\Delta t\\
=\int_a^t \partial_2[L]\left(\tau,\hat{x}(\tau),
\Delta[\hat{x}](\tau)\right) \Delta \tau h(t)|_{t=a}^{t=b}
-\int_a^b \left(\int_a^{\sigma (t)} \partial_2[L]\left(\tau,
\hat{x}(\tau),\Delta[\hat{x}](\tau)\right) \Delta \tau
\, \Delta[h](t)\right) \Delta t.
\end{multline*}
Because $h(a) =h(b)= 0$, the necessary condition \eqref{eq:FT} can be written as
\begin{equation*}
\int_a^b
\left(\partial_3[L]\left(t,\hat{x}(t),\Delta[\hat{x}](t)\right)
-\int_a^{\sigma(t)} \partial_2[L]\left(t,\hat{x}(t),\Delta[\hat{x}](t)\right)\Delta \tau \right)
\Delta[h](t) \, \Delta t = 0
\end{equation*}
for all $h\in C_{rd}^1$ such that $h(a) = h(b)=0$.  Thus, by the
Dubois--Reymond Lemma (see \cite[Lemma~4.1]{b7}), we have
\begin{equation*}
\partial_3[L]\left(t,\hat{x}(t),\Delta[\hat{x}](t)\right)
=\int_a^{\sigma(t)} \partial_2[L]\left(\tau,\hat{x}(\tau),\Delta[\hat{x}](\tau)\right)\Delta \tau  + c
\end{equation*}
for some $c\in \mathbb{R}$ and all  $t \in {\T}^\kappa$.
\end{proof}


\section{Conclusion}
\label{sec:lackCo}

Given a variational functional and a corresponding
Euler--Lagrange equation, the problem of coherence
concerns the coincidence of a direct embedding
of the given Euler--Lagrange equation
with the one obtained from the application
of the embedding to the variational functional
followed by application of the least action principle.
An embedding is not always coherent and a nontrivial problem is to find conditions
under which the embedding can be made coherent. An example of this is given
by the standard discrete embedding: the discrete embedding
of the Euler--Lagrange equation gives \eqref{eq:d:d:EL}
but the Euler--Lagrange equation \eqref{eq:nabla:EL}
obtained by the standard discrete calculus of variations does not coincide.
On the other hand, from the point of view of numerical integration
of ordinary differential equations,
we know that the discrete variational embedding is better
than the direct discrete embedding of the Euler--Lagrange equation
(\textrm{cf.} Example~\ref{ex2}).
The lack of coherence means that a pure algebraic discretization
of the Euler--Lagrange equation is not good in general,
because we miss some important dynamical properties
of the equation which are encoded in the Lagrangian functional.
A method to solve this default of coherence had been recently
proposed in \cite{BCGI}, and consists to rewrite
the classical Euler--Lagrange \eqref{eq:class:EL:pre}
as an asymmetric differential equation using left and right derivatives.
Inspired by the results of \cite{fmt}, here we propose a completely
different point of view to embedding based on the
Euler--Lagrange equation in integral form. For that we introduce the new
delta-integral embedding (\textrm{cf.} Definition~\ref{embeddingfunctional}).
Our main result shows that the delta-integral embedding and the delta-variational embedding
are coherent for any possible discretization (Theorem~\ref{thm:mr} is valid
on an arbitrary time scale).


\section*{Acknowledgements}

The second and third authors were partially supported by the
\emph{Systems and Control Group} of the R\&D Unit CIDMA through the
Portuguese Foundation for Science and Technology (FCT). Malinowska
was also supported by BUT Grant S/WI/2/11; Torres by the FCT research
project PTDC/MAT/113470/2009.


\medskip



\end{document}